\documentclass[12pt,reqno]{amsart}
\usepackage{latexsym}
\usepackage{amssymb}
\usepackage{mathrsfs}
\usepackage{amsmath}
\usepackage{fancybox,color}
\usepackage{enumerate}
\usepackage[latin1]{inputenc}
\usepackage{soul} 
\usepackage[colorlinks=true, linkcolor=blue, citecolor=blue]{hyperref}

\usepackage{enumitem}

\usepackage{pgf,tikz}
\usetikzlibrary{patterns,arrows,calc,decorations.pathreplacing}

\newcommand{\ch}[1]{{\mbox{\raise 1pt\hbox{\Large$\chi$}}}_{\lower 1pt\hbox{$\scriptstyle #1$}}}

\usepackage[left=2.3cm,top=2.5cm,right=2.3cm]{geometry}

\def\1{\raisebox{2pt}{\rm{$\chi$}}}

\newtheorem{theorem}{Theorem}[section]

\newtheorem{lemma}[theorem]{Lemma}

\theoremstyle{definition}
\newtheorem{definition}[theorem]{Definition}

\theoremstyle{definition}

\newcommand{\R}{{\mathbb R}}

\newcommand{\N}{{\mathbb N}}
\newcommand{\Z}{{\mathbb Z}}

\newcommand\diam{\operatorname{diam}}

\DeclareMathOperator{\ucodima}{\overline{co\;dim}_A}
\DeclareMathOperator{\lcodima}{\underline{co\;dim}_A}

\def\cprime{$'$}

{\catcode`p =12 \catcode`t =12 \gdef\eeaa#1pt{#1}}      
\def\accentadjtext#1{\setbox0\hbox{$#1$}\kern   
                \expandafter\eeaa\the\fontdimen1\textfont1 \ht0 }
\def\accentadjscript#1{\setbox0\hbox{$#1$}\kern 
                \expandafter\eeaa\the\fontdimen1\scriptfont1 \ht0 }
\def\accentadjscriptscript#1{\setbox0\hbox{$#1$}\kern   
                \expandafter\eeaa\the\fontdimen1\scriptscriptfont1 \ht0 }
\def\accentadjtextback#1{\setbox0\hbox{$#1$}\kern       
                -\expandafter\eeaa\the\fontdimen1\textfont1 \ht0 }
\def\accentadjscriptback#1{\setbox0\hbox{$#1$}\kern     
                -\expandafter\eeaa\the\fontdimen1\scriptfont1 \ht0 }
\def\accentadjscriptscriptback#1{\setbox0\hbox{$#1$}\kern 
                -\expandafter\eeaa\the\fontdimen1\scriptscriptfont1 \ht0 }
\def\itoverline#1{{\mathsurround0pt\mathchoice
        {\rlap{$\accentadjtext{\displaystyle #1}
                \accentadjtext{\vrule height1.593pt}
                \overline{\phantom{\displaystyle #1}
                \accentadjtextback{\displaystyle #1}}$}{#1}}
        {\rlap{$\accentadjtext{\textstyle #1}
                \accentadjtext{\vrule height1.593pt}
                \overline{\phantom{\textstyle #1}
                \accentadjtextback{\textstyle #1}}$}{#1}}
        {\rlap{$\accentadjscript{\scriptstyle #1}
                \accentadjscript{\vrule height1.593pt}
                \overline{\phantom{\scriptstyle #1}
                \accentadjscriptback{\scriptstyle #1}}$}{#1}}
        {\rlap{$\accentadjscriptscript{\scriptscriptstyle #1}
                \accentadjscriptscript{\vrule height1.593pt}
                \overline{\phantom{\scriptscriptstyle #1}
                \accentadjscriptscriptback{\scriptscriptstyle #1}}$}{#1}}}}

\newcommand{\iol}{\itoverline}

\def\1{\raisebox{2pt}{\rm{$\chi$}}}

\newcommand{\Lip}{\operatorname{Lip}}

\def\vint_#1{\mathchoice%
        {\mathop{\kern 0.2em\vrule width 0.6em height 0.69678ex depth -0.58065ex
                \kern -0.8em \intop}\nolimits_{\kern -0.4em#1}}%
        {\mathop{\kern 0.1em\vrule width 0.5em height 0.69678ex depth -0.60387ex
                \kern -0.6em \intop}\nolimits_{#1}}%
        {\mathop{\kern 0.1em\vrule width 0.5em height 0.69678ex
            depth -0.60387ex
                \kern -0.6em \intop}\nolimits_{#1}}%
        {\mathop{\kern 0.1em\vrule width 0.5em height 0.69678ex depth -0.60387ex
                \kern -0.6em \intop}\nolimits_{#1}}}
\def\vintslides_#1{\mathchoice%
        {\mathop{\kern 0.1em\vrule width 0.5em height 0.697ex depth -0.581ex
                \kern -0.6em \intop}\nolimits_{\kern -0.4em#1}}%
        {\mathop{\kern 0.1em\vrule width 0.3em height 0.697ex depth -0.604ex
                \kern -0.4em \intop}\nolimits_{#1}}%
        {\mathop{\kern 0.1em\vrule width 0.3em height 0.697ex depth -0.604ex
                \kern -0.4em \intop}\nolimits_{#1}}%
        {\mathop{\kern 0.1em\vrule width 0.3em height 0.697ex depth -0.604ex
                \kern -0.4em \intop}\nolimits_{#1}}}

\newcommand{\dist}{\operatorname{dist}}

\title[Characterization of the upper Assouad codimension]{Characterization of the upper Assouad codimension}

\author[J. Kline]{Josh Kline}
\address[J.K.]{Department of Mathematical Sciences, P.O. Box 210025, University of Cincinnati, Cincinnati, OH 45221-0025, U.S.A}
 \email{klinejp@mail.uc.edu}

\author[A. V. V\"ah\"akangas]{Antti V. V\"ah\"akangas}
\address[A.V.V.]{Department of Mathematics and Statistics, P.O. Box 35, FI-40014 University of Jyvaskyla, Finland}
 \email{antti.vahakangas@iki.fi}

\makeatletter
\@namedef{subjclassname@2020}{{\mdseries 2020} Mathematics Subject Classification}
\makeatother

\date{\today}

\begin{document}

\begin{abstract}
We present a new notion, the upper Aikawa codimension, and establish its equivalence with the upper Assouad codimension in a metric space  with a doubling measure. To achieve this result, we first prove variant of a local fractional Hardy inequality.
\end{abstract}

\maketitle

\section{Introduction}

The focus of this note is a variation of the Assouad dimension, specifically, the upper Assouad codimension. Alongside its dual counterpart, known as the lower Assouad codimension, these concepts were introduced in  \cite{KaenmakiLehrbackVuorinen2013} for a metric space with a doubling measure. Their applications extend for instance to Hardy inequalities \cite{MR3237044,MR3631460}, with distance measured to an irregular set, and density properties of compactly supported smooth functions in fractional Sobolev spaces \cite{MR4454384}. For more comprehensive treatments of Assouad dimensions, additional information can be found in \cite{MR4237248} and the monographs \cite{MR4411274,MR4306765}.

Another notion of dimension, defined through the integral condition
\[
\vint_{B(z,R)} \frac{1}{\dist(y,E)^{\alpha}}\,dy\le CR^{-\alpha}
\]
 was employed by Aikawa in \cite{MR1143471} to investigate quasiadditivity properties of Riesz capacity; we also refer to \cite{MR1439503}.
The following result from \cite[Theorem 5.1]{MR3055588} establishes
a connection between these notions of dimensions:
The lower Assouad codimension (refer to Definition~\ref{d.lowerassouad}) of a set $E\subset X$ is the supremum of those $\alpha\ge 0$ for which $E$ satisfies the lower Aikawa condition (Definition~\ref{d.loweraikawa}). 
Consequently, the lower Assouad codimension $\lcodima(E)$ can be referred to as the lower Aikawa codimension of $E$.

The aim of this note is to prove an analogous characterization for the upper Assouad codimension (see Definition \ref{d.upperassouad}), which is the dual concept of the lower Assouad codimension. More specifically, when $E\subset X$, we prove that $\ucodima(E)$ is the infimum of all $\alpha>0$ for which $E$ satisfies a novel upper Aikawa condition (Definition~\ref{d.upperaikawa}). 
This main result is  formulated in Theorem~\ref{t.main}, and to our knowledge, it is new even in $\R^n$ equipped with the Euclidean distance and the Lebesgue measure. 

The main challenge concerning the proof of Theorem~\ref{t.main} lies in establishing  the inequality $\ucodima(E)\le \inf A$; here, $A$ denotes the set of exponents $\alpha>0$ for which $E$ satisfies the upper Aikawa condition. 
This inequality is deduced from a local Hardy inequality, the fractional version of which we prove  under the assumption that $E$ satisfies the upper Aikawa condition with exponent $\alpha$ (refer to Section \ref{e.localH}). 
In Section \ref{e.truncation} we lay the groundwork for the local Hardy inequality by employing a truncation technique from \cite{MR946438} on the set $E$. We also utilize the fractional variant \cite{MR3331699} of the Maz'ya truncation method \cite[p. 110]{Mazya1985} in the proof. 

\subsection*{Acknowledgements}

We appreciate the utilization of ChatGPT version 3.5 in the final preparation of this note.

\section{Notation and Preliminaries}\label{s.notation}

We assume throughout this paper that $X=(X,d,\mu)$ is 
a metric measure space (with at least two points),  
where
$\mu$ is a Borel measure supported on $X$ such that 
$0<\mu(B)<\infty$ for
all (open) balls 
\[B=B(x,r):= \{y\in X : d(x,y)< r\}\]
with $x\in X$ and $r>0$. 
We also assume throughout
that $\mu$ is \emph{doubling}; that is, there is a constant $c_\mu >  1$ such that 
whenever $x\in X$ and $r>0$, we have
\begin{equation}\label{e.doubling}
  \mu(B(x,2r))\le c_\mu\, \mu(B(x,r)).
\end{equation}
By iterating the doubling condition \eqref{e.doubling}, we can find constants $s=\log_2 c_\mu>0$ and $C>0$ such that
\begin{equation}\label{eq:QDoublingDim}
\frac{\mu(B(x,r))}{\mu(B(x,R))}\ge C\Bigl(\frac rR\Bigr)^s
\end{equation}
whenever $x\in X$ and $0<r\le R<\diam(X)$; see \cite[Lemma~3.3]{MR2867756}.

The closure of a set $E\subset X$ is denoted by $\iol{E}$.
We use $\dist(x,E)=\inf\{d(x,y)\,:\,y\in E\}$ to denote the distance from a point $x\in X$ to a nonempty set $E\subset X$.
If $E\subset X$, then $\mathbf{1}_{E}$
denotes the characteristic function of $E$; that is, $\mathbf{1}_{E}(x)=1$ if $x\in E$
and $\mathbf{1}_{E}(x)=0$ if $x\in X\setminus E$.
We abbreviate $dy=d\mu(y)$ in the integrals and 
use the following familiar notation \[
u_A=\vint_{A} u(y)\,dy=\frac{1}{\mu(A)}\int_A u(y)\,d\mu(y)
\]
for the integral average of $u\in L^1(A)$ over a measurable set $A\subset X$
with $0<\mu(A)<\infty$.
 The set of all Lipschitz functions $u:X\to \R$ is denoted by $\Lip(X)$. 
We denote $\N=\{1,2,3,\ldots\}$ and $\N_0=\N\cup\{0\}$.

\subsection{Lower Assouad codimension and lower Aikawa condition}

In order to motivate and provide context for our novel
upper Aikawa condition in Section \ref{s.upper}, we first recall 
the lower  Assouad codimension 
 \cite{KaenmakiLehrbackVuorinen2013} and its
 relation to the lower Aikawa condition  \cite{MR3055588}.
 
  When $E\subset X$ and $r>0$, the
open $r$-neighbourhood of $E$
is the set 
\[E_r=\{x\in X \,:\, \dist(x,E)<r\}.\]
Observe that $E_r=\iol{E}_r$.  That is, the
$r$-neighbourhoods of $E$ and its closure coincide.

\begin{definition}\label{d.lowerassouad}
The \emph{lower Assouad codimension} of a nonempty set $E\subset X$,
denoted by $\lcodima(E)$, is
the supremum of all $Q\ge 0$ for
which there exists a constant $C>0$ such that
\begin{equation}\label{e.lascd}
\frac{\mu(E_r\cap B(z,R))}{\mu(B(z,R))}\le C\Bigl(\frac{r}{R}\Bigr)^Q
\end{equation}
for every $z\in E$ and all $0<r<R<\diam(E)$. 
If $\diam(E)=0$, then
the restriction $R<\diam(E)$ is replaced
with $R<\diam(X)$.
\end{definition}

\begin{definition}\label{d.loweraikawa}
Suppose $E\subset X$ is a nonempty set and $\alpha\ge 0$. Assume that there exists a constant $C>0$
such that
\begin{equation}\label{e.aikawa}
\vint_{B(z,R)} \frac{1}{\dist(y,E)^{\alpha}}\,dy\le CR^{-\alpha}
\end{equation}
for every $z\in E$ and $0<R<\diam(E)$. If $\diam(E)=0$, then the
restriction $R<\diam(E)$ is replaced with $R<\diam(X)$. 
Then we say that $E$ satisfies the \emph{lower Aikawa condition with exponent $\alpha$. }
\end{definition}

Observe that \eqref{e.aikawa} always holds for $\alpha=0$.
The following result is \cite[Theorem 5.1]{MR3055588}.

\begin{theorem}\label{e.basic}
Let $E\subset X$ be a nonempty set. Then $\lcodima(E)$
is the supremum of all $\alpha\ge 0$ for which $E$ satisfies
the lower Aikawa condition with exponent $\alpha$.
\end{theorem}

Our  definition of the upper Aikawa condition, see  Definition \ref{d.upperaikawa} below, is
motivated by  condition (B) of the following Theorem~\ref{t.self_aikawa}.

We recall that a set $E\subset\R^n$ is \emph{porous}\index{porous set} if there exists a constant $C>0$ such that
for every $x\in\R^n$ and $r>0$ there exists $y\in\R^n$ satisfying $B(y,Cr)\subset B(x,r)\setminus E$.

\begin{theorem}\label{t.self_aikawa}
Consider $X=\R^n$ equipped with the
Euclidean distance and the Lebesgue measure denoted by $\mu$.
Let $0\le \alpha<n$ and $E\subset\R^n$
be a nonempty set. Let us consider the following conditions:
\begin{itemize}
\item[(A)]
$E$ satisfies
the lower Aikawa condition with exponent $\alpha$.
\item[(B)] For every $\varepsilon>0$, there exists $\delta>0$ such that
\[
\vint_{B(z,R)} \frac{\mathbf{1}_K(y)}{\dist(y,E)^\alpha}\,dy \le \varepsilon R^{-\alpha}
\]
whenever $z\in E$, $R>0$ and $K\subset\R^n$
is a Lebesgue measurable set such that \[
\mu(K\cap B(z,R))\le \delta \mu(B(z,R))\,.\] 
\end{itemize}
Condition  (A) implies  (B). Conversely, if $E$ is porous and $\dist(\cdot,E)^{-\alpha}$ is locally integrable, then
 (B) implies  (A).
\end{theorem}

\begin{proof}
Let $w(y)=\dist(y,E)^{-\alpha}$ for every $y\in\R^n$.
Let us assume that condition (A) holds and fix $\varepsilon>0$.
By \cite[Lemma 10.4]{MR4306765}, we may 
further assume that \eqref{e.aikawa} holds for all $R>0$
with a possibly different constant $C_1>0$.
Then from \cite[Theorem 10.9]{MR4306765}, we
conclude that $w$ belongs to the Muckenhoupt class $A_\infty$. By applying one of the well-known characterizations  of $A_\infty$ weights \cite{Muckenhoupt1974}, we infer that there
exists $\delta>0$ as follows.
Fix $z\in E$, $R>0$ and a Lebesgue measurable set $K\subset \R^n$
such that $\mu(K\cap B(z,R))\le \delta \mu(B(z,R))$.
Then
\begin{align*}
\int_{B(z,R)} \frac{\mathbf{1}_K(y)}{\dist(y,E)^\alpha}\,dy&=\int_{K\cap B(z,R)} w(y)\,dy\le \frac{\varepsilon}{C_1} \int_{B(z,R)} w(y)\,dy\\&=\frac{\varepsilon}{C_1} \int_{B(z,R)} \frac{1}{\dist(y,E)^\alpha}\,dy\le \varepsilon R^{-\alpha} \mu(B(z,R))\,,
\end{align*}
where the last step follows from \eqref{e.aikawa}.
Condition (B) follows from this.

Conversely, assume that  (B) holds, the set $E$ is porous and $w$ is locally integrable. Denote $w(A)=\int_A w(y)\,dy$ for  measurable sets $A\subset \R^n$.
We  show that $w$ belongs to Muckenhoupt $A_\infty$ class of weights. Since $E$ is porous, condition (A) then follows by applying \cite[Theorem 10.15]{MR4306765}.
For this purpose, we let $\varepsilon>0$.
By   \cite{Muckenhoupt1974}, it suffices to find a $\delta>0$ such that
$w(K)\le \varepsilon w(B(x,r))$ whenever
$B(x,r)\subset \R^n$ is a ball and $K\subset B(x,r)$
is a measurable set such that $\mu(K)\le \delta\mu(B(x,r))$.
We  choose $\delta>0$ such that
condition (B) holds for the pair $\varepsilon/4^n>0$ and $\delta>0$.
Clearly, we may further assume that $4^\alpha\delta<\varepsilon$.

First we assume that $2r\le \dist(x,E)$. Observe
that $\dist(x,E)/2\le \dist(y,E)\le 2\dist(x,E)$ for every $y\in B(x,r)$, and therefore
\begin{equation}\label{e.case1}
\begin{split}
\int_{B(x,r)} \frac{\mathbf{1}_K(y)}{\dist(y,E)^\alpha}\,dy
&\le 2^\alpha\dist(x,E)^{-\alpha}\mu(K)\\&\le  (2\dist(x,E))^{-\alpha} 4^\alpha\delta \mu(B(x,r))
\le {\varepsilon}\int_{B(x,r)} \frac{1}{\dist(y,E)^\alpha}\,dy\,.
\end{split}
\end{equation}

Next we consider the case, where $2r> \dist(x,E)$. Write $R=2r$ and
choose $z\in E$ such that $x\in B(z,R)$. Observe
that $B(x,r)\subset B(z,2R)$ and 
\[
\mu(K\cap B(z,2R))\le \mu(K)\le \delta\mu(B(x,r))\le \delta \mu(B(z,2R))\,. 
\]
Condition (B) and the fact that $\mu$ is the Lebesgue measure imply that
\begin{equation}\label{e.case2}
\begin{split}
\int_{B(x,r)} \frac{\mathbf{1}_K(y)}{\dist(y,E)^\alpha}\,dy&\le \int_{B(z,2R)} \frac{\mathbf{1}_K(y)}{\dist(y,E)^\alpha}\,dy \le \frac{\varepsilon}{4^n} (2R)^{-\alpha}\mu(B(z,2R))\\
&\le \varepsilon(3r)^{-\alpha}\mu(B(x,r))\le  {\varepsilon}\int_{B(x,r)} \frac{1}{\dist(y,E)^\alpha}\,dy\,.
\end{split}
\end{equation}

In any case, by combining inequalities \eqref{e.case1} and \eqref{e.case2}, we find that
\[
w(K)=w(K\cap B(x,r))\le 
 \varepsilon w(B(x,r))\,,
\]
as claimed.
\end{proof}

\subsection{Upper Assouad codimension and upper Aikawa condition}\label{s.upper}

We first recall  the  upper  Assouad codimension from
 \cite{KaenmakiLehrbackVuorinen2013}.

 \begin{definition}\label{d.upperassouad}
The \emph{upper Assouad codimension} of a nonempty set $E\subset X$,
denoted by $\ucodima(E)$, is
the infimum of all $Q\ge 0$ for
which there exists a constant $c>0$ such that
\begin{equation}\label{e.uascd}
\frac{\mu(E_r\cap B(z,R))}{\mu(B(z,R))}\ge c\Bigl(\frac{r}{R}\Bigr)^Q
\end{equation}
for every $z\in E$ and all $0<r<R<\diam(E)$. 
If $\diam(E)=0$, then
the restriction $R<\diam(E)$ is replaced
with $R<\diam(X)$.
\end{definition}

Observe from Inequality~\eqref{eq:QDoublingDim}
that $\ucodima(E)\le \log_2 c_\mu$ for every nonempty $E\subset X$.

In this note, we introduce the following concept that we call
an upper Aikawa condition. This terminology is motivated
by Theorem \ref{t.self_aikawa}.

\begin{definition}\label{d.upperaikawa}
Suppose $E\subset X$ is a nonempty set and $\alpha>0$. Assume that, for every $\varepsilon>0$, there exists
$\delta>0$ such that
\[
 \vint_{B(z,R)} \frac{\mathbf{1}_{K}(y)}{\dist(y,E)^{\alpha}}\,dy\ge \frac{R^{-\alpha}}{\varepsilon}
\]
whenever $z\in E$, $0<R<\diam(E)$ and $K\subset X$ is a Borel set that satisfies the condition
\[
\mu(B(z,R)\setminus K)\le\delta \mu(B(z,R))\,.
\]
If $\diam(E)=0$, then the restriction $R<\diam(E)$ is replaced
with $R<\diam(X)$.
Then we say that $E$ satisfies the \emph{upper Aikawa condition with exponent $\alpha$. }
\end{definition}

\section{Truncation of the set}\label{e.truncation}


In order to prove the local Hardy inequality in Section \ref{e.localH}, 
we  need to truncate a given closed set $E$ to a ball in a manner which preserves big pieces of the original set $E$. Therefore, we use
the following method of truncation from \cite{MR946438}, see also
\cite{MR1386213,MR1869615,MR3673660}. 
The proofs of Lemma~\ref{l.truncation} and Lemma~\ref{l.pallot} can be found from \cite{MR3673660}; see also \cite[Section 7.6]{MR4306765} for the case $X=\R^n$.

\begin{lemma}\label{l.truncation}
Assume that $E\subset X$ is a closed set and let $z\in E$ and $r>0$.
Let $F_0=E\cap \iol{B(z,\frac r 2)}$ and define recursively
\[
F_j=\bigcup_{x\in F_{j-1}} E\cap \iol{B(x,2^{-j-1}r)}\,,\qquad \text{ for every }j\in\N\,.
\]
Let $F=\overline{\bigcup_{j\in \N_0} F_j}$.
Then 
$z\in F\subset E\cap\iol{B(z,r)}$, and
$F_{j-1}\subset F_{j}\subset F$ for every $j\in\N$.
\end{lemma}

The truncated set $F$ in Lemma~\ref{l.truncation} has the property
that there always exist certain balls whose intersection with $F$
contain big pieces of the original set $E$.

\begin{lemma}\label{l.pallot}
Assume that $E$, $B(z,r)$ and $F$ are as in Lemma~{\textup{\ref{l.truncation}}}.
Let $m\in\N_0$ and $x\in X$ be such that
$\dist(x,F) < 2^{-m+1}r$, and write $r_m=2^{-m-1}r$. 
There exists $y_{x,m}\in E$ such that 
$B(y_{x,m},r_m)\subset B(x,8r_m)$ 
and 
\begin{equation}\label{e.bubble} 
E\cap \iol{B(y_{x,m},\tfrac 1 2{r_m})} = F \cap \iol{B(y_{x,m},\tfrac 1 2{r_m})}.
\end{equation}
\end{lemma}

The following lemma shows that a local variant
of the upper Aikawa condition persists to hold for the truncated set $F$.

\begin{lemma}\label{l.bigpiece}
Let $\alpha>0$ and assume that a closed set $E\subset X$ satisfies the upper Aikawa condition with exponent $\alpha$.
Then for every $\varepsilon >0$ there exists $\delta>0$ as follows.
Assume that $B(z,r)$ and $F$ are as in Lemma~{\textup{\ref{l.truncation}}}.
Let $m\in\N_0$ and $x\in X$ be such that
$\dist(x,F) < 2^{-m+1}r$ and either $2^{-m-3}r<\diam(E)$ if $\diam(E)>0$, or $2^{-m-3}r<\diam(X)$ if $\diam(E)=0$.Write $r_m=2^{-m-1}r$. 
Assume also that $K\subset X$ is a Borel set such that
\[
\vint_{B(x,8r_m)} \mathbf{1}_{X\setminus K}(y)\,dy\le \delta\,.
\]
Then
\[
r_m^{-\alpha}\le \varepsilon\vint_{B(x,8r_m)} \frac{\mathbf{1}_{K}(y)}{\dist(y,F)^{\alpha}}\,dy\,.
\]
\end{lemma}

\begin{proof}
Fix $\varepsilon>0$. 
Let $\varepsilon'=c_\mu^{-6}\varepsilon>0$.
Choose $\delta'>0$ such that $E$ satisfies the upper Aikawa condition with exponent $\alpha$, for the pair $\varepsilon'>0$ and $\delta'>0$. Choose $\delta=c_\mu^{-6}\delta'$.
Write $B=B(x,8r_m)$ and $B'={B}(y_{x,m},r_m/4)$, where
$y_{x,m}\in E$ is as in Lemma \ref{l.pallot}.  
Recall from Lemma \ref{l.pallot}
that $B'\subset B$.
Using also \eqref{e.doubling} and the assumption, it follows that
\begin{align*}
\vint_{B'} \mathbf{1}_{X\setminus K}(y)\,dy\le
c_\mu^6\vint_{B} \mathbf{1}_{X\setminus K}(y)\,dy\le c_\mu^6\delta=\delta'\,.
\end{align*}
Hence, 
\[
\mu(B'\setminus K)\le \delta'\mu(B')\,.
\]
By  \eqref{e.bubble}, we have $\dist(y,E)\ge \dist(y,F)$ if $y\in B'={B}(y_{x,m},r_m/4)$.
Using also the inclusion $B'\subset B$ and the upper Aikawa condition  for the pair $\varepsilon'$ and $\delta'$, 
we obtain
\[
r_m^{-\alpha}\le (r_m/4)^{-\alpha}\le \varepsilon'\vint_{B'} \frac{\mathbf{1}_{K}(y)}{\dist(y,E)^\alpha}\,dy
\le \varepsilon'\vint_{B'} \frac{\mathbf{1}_{K}(y)}{\dist(y,F)^\alpha}\,dy
\le c_\mu^6 \varepsilon'\vint_{B} \frac{\mathbf{1}_{K}(y)}{\dist(y,F)^\alpha}\,dy
\,.
\]
Since $c_\mu^6\varepsilon'=\varepsilon$, this implies the desired estimate.
\end{proof}

By truncating the set $E$, we are able to obtain 
the following  holefilling lemma. This is an adaptation of 
\cite[Lemma 5.3]{MR3673660}, see also \cite[Lemma 7.27]{MR4306765}.

\begin{lemma}\label{l.absorb}
Let $1\le p<\infty$, $0<s<1$
and assume that $E$, $B(z,r)$ and $F$ are as in Lemma~{\textup{\ref{l.truncation}}}.
Suppose that $\sigma\ge 1$ and $\kappa\ge 2$, and $u\in \Lip(X)$ satisfies $u=0$ in $F$. Then
\begin{equation}\label{e.holefill}
\begin{split}
&\int_{B(z,\sigma \kappa r)\setminus F} \frac{\lvert u(x)\rvert^{p}}{\dist(x,F)^{sp}}\,dx\\
&\qquad\le 
 C_1\int_{B(z,\sigma\kappa r)}\int_{B(z,\sigma\kappa r)} \frac{\lvert u(x)-u(y)\rvert^p}{d(x,y)^{sp}\, \mu(B(x,d(x,y)))}\,dy\,dx+
 C_2\int_{B(z,\kappa r)\setminus F} \frac{\lvert u(x)\rvert^{p}}{\dist(x,F)^{sp}}\,dx\,,
\end{split}
\end{equation}
where $C_1=C(s,c_\mu,p,\sigma,\kappa)$ and $C_2=1+3^p \kappa^{sp} C(c_\mu,\sigma)$.
\end{lemma}

\begin{proof}
Since $F\subset \iol{B(z,r)}$ and $\kappa\ge 2$, we have 
\begin{align*}
&\int_{(B(z,\sigma\kappa r)\setminus F)\setminus B(z,\kappa r)} \frac{\lvert u(x)\rvert^{p}}{\dist(x,F)^{sp}}\,dx
\le r^{-sp}\int_{B(z,\sigma\kappa r)}\lvert u(x)\rvert^p\,dx\\
& \qquad \le 3^p r^{-sp} \biggl( \int_{B(z,\sigma\kappa r)}\bigl\lvert u(x)-u_{B(z,\sigma\kappa r)}\bigr\rvert^p\,dx 
\\
&\qquad\qquad+ \mu(B(z,\sigma\kappa r))\bigl\lvert u_{B(z,\sigma\kappa r)}-u_{B(z,\kappa r)}\bigr\rvert^p + 
\mu(B(z,\sigma\kappa r)) \bigl\lvert u_{B(z,\kappa r)}\bigr\rvert^p \biggr).
\end{align*}
By applying the  H\"older's inequality and the fractional $(s,p,p,p)$-Poincar\'e inequality from \cite[Lemma 2.2]{DLV2021}, we get
\begin{align*}
& 3^p r^{-sp} \biggl(\int_{B(z,\sigma\kappa r)} \bigl\lvert u(x)-u_{B(z,\sigma\kappa r)}\bigr\rvert^p\,dx + 
  \mu(B(z,\sigma\kappa r))\bigl\lvert u_{B(z,\sigma\kappa r)}-u_{B(z,\kappa r)}\bigr\rvert^p\biggr) \\
& \qquad \le C(c_\mu,\sigma,p)r^{-sp}\int_{B(z,\sigma\kappa r)}\bigl\lvert u(x)-u_{B(z,\sigma\kappa r)}\bigr\rvert^p\,dx\\
 &  \qquad\le C(s,c_\mu,p,\sigma,\kappa)\int_{B(z,\sigma\kappa r)}\int_{B(z,\sigma\kappa r)} \frac{\lvert u(x)-u(y)\rvert^p}{d(x,y)^{sp}\, \mu(B(x,d(x,y)))}\,dy\,dx.
\end{align*}
On the other hand, since  $u=0$ in $F$,  we obtain 
\begin{align*}
3^pr^{-sp} \mu(B(z,\sigma\kappa r)) \bigl\lvert u_{B(z,\kappa r)}\bigr\rvert^p 
 &\le 3^p C(c_\mu,\sigma)r^{-sp}\int_{B(z,\kappa r)\setminus F}\lvert u(x)\rvert^p\,dx\\
& \le 3^p \kappa^{sp} C(c_\mu,\sigma)\int_{B(z,\kappa r)\setminus F}\frac{\lvert u(x)\rvert^p}{\dist(x,F)^{sp}}\,dx\,.
\end{align*}
Since
\begin{align*}
\int_{B(z,\sigma \kappa r)\setminus F} \frac{\lvert u(x)\rvert^{p}}{\dist(x,F)^{sp}}\,dx
=
\int_{B(z,\kappa r)\setminus F} \frac{\lvert u(x)\rvert^{p}}{\dist(x,F)^{sp}}\,dx+\int_{(B(z,\sigma\kappa r)\setminus F)\setminus B(z,\kappa r)} \frac{\lvert u(x)\rvert^{p}}{\dist(x,F)^{sp}}\,dx\,,
\end{align*}
the claim follows from the above estimates.
\end{proof}

\section{Local Hardy inequality}\label{e.localH}

We show
that the upper Aikawa condition implies a local variant of the Hardy inequality.
The proof combines various techniques. For instance, we borrow the following
fractional variant of the  Maz'ya truncation method from \cite{MR3331699}.

\begin{lemma}\label{l.mazya}
Let $u\in\Lip(X)$. For every $k\in\Z$, define
a function $u_k:X\to [0,1]$ by
\[
u_k(x) = \begin{cases}
1, & \text{if $|u(x)| \geq 2^{k+1}$,}\\
\frac{|u(x)|}{2^k}-1, &\text{if $2^k < |u(x)| < 2^{k+1}$,}\\
0, & \text{if $|u(x)| \leq 2^k$.}
\end{cases}
\]
Then there exists a constant $C=C(p)$ such that
\begin{align*}
\sum_{k\in \Z} 2^{kp}\int_{B(z,R)}\int_{B(z,R)}&\frac{\lvert u_k(x)-u_k(y)\rvert^p}{d(x,y)^{sp}\, \mu(B(x,d(x,y)))}\,dy\,dx \\
&\qquad \le C\int_{B(z,R)}\int_{B(z,R)}\frac{\lvert u(x)-u(y)\rvert^p}{d(x,y)^{sp}\, \mu(B(x,d(x,y)))}\,dy\,dx
\end{align*}
whenever $z\in X$ and $R>0$.
\end{lemma}

\begin{proof}
For every $i\in \Z$ we
denote
\[
 E_i = \{x\in B(z,R) : |u(x)| > 2^i \} \quad \text{and} \quad A_i = E_i \setminus E_{i+1}\,.
\]
Fix $k\in\Z$ and observe that
\begin{equation}\label{e.decomp}
\begin{split}
B(z,R)&= \{x\in B(z,R)\,:\, 0\le \lvert u(x)\rvert<\infty\} \\&= \underbrace{\{x\in B(z,R) \,:\, u(x)=0\}}_{=:N}\cup \bigcup_{i\in \Z} A_i\\
&=N  \cup  \bigcup_{i<k} A_i  \cup A_k \cup \bigcup_{i>k} A_i\,.
\end{split}
\end{equation}
By using the decomposition \eqref{e.decomp} and combining terms, we obtain that
\begin{equation}\label{e.plugL}
\begin{split}
&\int_{B(z,R)} \int_{B(z,R)} \frac{\lvert u_k(x)-u_k(y)\rvert^p}{d(x,y)^{sp}\, \mu(B(x,d(x,y)))}\,dy\,dx\\
&\qquad \leq
\bigg\{  \sum_{i\leq k} \sum_{j\geq k} \int_{A_i} \int_{A_j} + \sum_{j\ge k} \int_{N} \int_{A_j} \bigg\}\frac{\lvert u_k(x)-u_k(y)\rvert^p}{d(x,y)^{sp}\, \mu(B(x,d(x,y)))}\,dy\,dx\\
&\qquad  \qquad +\bigg\{  \sum_{i\geq k} \sum_{j\leq k} \int_{A_i} \int_{A_j} + \sum_{i\ge k} \int_{A_i} \int_{N} \bigg\}\frac{\lvert u_k(x)-u_k(y)\rvert^p}{d(x,y)^{sp}\, \mu(B(x,d(x,y)))}\,dy\,dx\,.
\end{split}
\end{equation}
We observe that $|u_k(x)-u_k(y)| \leq 2^{-k} |u(x)-u(y)|$. Moreover, if $x\in A_i$ and $y\in A_j$,
where $i+2\leq j$, then
$ |u(x)-u(y)| \geq |u(y)| - |u(x)| \geq 2^{j-1}$,
hence $|u_k(x) - u_k(y)| \leq 1 \leq 2\cdot 2^{-j} |u(x)-u(y)|$.
Therefore,
\[
 |u_k(x)-u_k(y)| \leq 2\cdot 2^{-j} |u(x)-u(y)|, \quad (x,y)\in  A_i\times A_j\,,
\]
whenever $i\leq k \leq j$. 
Observe also that $\sum_{k=i}^j 2^{(k-j)p} < \sum_{k=-\infty}^j 2^{(k-j)p}  \le \frac{1}{1-2^{-p}}$. Hence,
\begin{align*}
\sum_{k\in \Z} 2^{kp}  \sum_{i\leq k} \sum_{j\geq k} & \int_{A_i} \int_{A_j} \frac{\lvert u_k(x)-u_k(y)\rvert^p}{d(x,y)^{sp}\, \mu(B(x,d(x,y)))}\,dy\,dx\\
& \le  2^{p}\sum_{k\in \Z} \sum_{i\leq k} \sum_{j\geq k} 2^{(k-j)p} \int_{A_i} \int_{A_j} \frac{\lvert u(x)-u(y)\rvert^p}{d(x,y)^{sp}\, \mu(B(x,d(x,y)))}\,dy\,dx\\
&\le  2^{p} \sum_{i\in\Z} \sum_{j\in\Z} \sum_{i\le k\le j}2^{(k-j)p} \int_{A_i} \int_{A_j} \frac{\lvert u(x)-u(y)\rvert^p}{d(x,y)^{sp}\, \mu(B(x,d(x,y)))}\,dy\,dx\\
&\le \frac{2^p}{1-2^{-p}}\int_{B(z,R)} \int_{B(z,R)} \frac{\lvert u(x)-u(y)\rvert^p}{d(x,y)^{sp}\, \mu(B(x,d(x,y)))}\,dy\,dx.
\end{align*}
A similar argument shows that
\begin{align*}
&\sum_{k\in \Z} 2^{kp}\sum_{j\ge k} \int_{N} \int_{A_j} \frac{\lvert u_k(x)-u_k(y)\rvert^p}{d(x,y)^{sp}\, \mu(B(x,d(x,y)))}\,dy\,dx\\
&\qquad \qquad  \le  2^{p} \sum_{k\in \Z} \sum_{j\ge k}  2^{(k-j)p} \int_{N} \int_{A_j} \frac{\lvert u(x)-u(y)\rvert^p}{d(x,y)^{sp}\, \mu(B(x,d(x,y)))}\,dy\,dx\\
&\qquad \qquad\le \frac{2^p}{1-2^{-p}}\int_{B(z,R)} \int_{B(z,R)} \frac{\lvert u(x)-u(y)\rvert^p}{d(x,y)^{sp}\, \mu(B(x,d(x,y)))}\,dy\,dx\,.
\end{align*}
Fubini's theorem and similar estimates for the two remaining series in \eqref{e.plugL} shows that 
\begin{align*}
&\sum_{k\in \Z} 2^{kp}\int_{B(z,R)}\int_{B(z,R)} \frac{\lvert u_k(x)-u_k(y)\rvert^p}{d(x,y)^{sp}\, \mu(B(x,d(x,y)))}\,dy\,dx\\
&\qquad  \leq
 \frac{2^{p+2}}{1-2^{-p}}\int_{B(z,R)} \int_{B(z,R)} \frac{\lvert u(x)-u(y)\rvert^p}{d(x,y)^{sp}\, \mu(B(x,d(x,y)))}\,dy\,dx\,.
\end{align*}
Thus, we obtain the desired inequality with $C = \frac{2^{p+2}}{1-2^{-p}}$.
\end{proof}

The following theorem gives a local fractional Hardy inequality
for sets satisfying the upper Aikawa condition with suitable exponent.

\begin{theorem}\label{p.maz'ya_local}
Let $1<q<p<\infty$ and $0<s<1$.
Suppose that a closed set $E\subset X$ satisfies the
upper Aikawa condition with exponent $\alpha=sq$.
Assume that $B(z,r)$ and $F$ are as in Lemma \ref{l.truncation}, with  either  $0<r<\diam(E)$ if
$\diam(E)>0$, or $0<r<\diam(X)$ if $\diam(E)=0$.
Then there exists a constant $C>0$ independent of $z$ and $r$ such that
\[
\int_{B(z,8r)\setminus F}  \frac{\lvert u(x)\rvert^p}{\dist(x,F)^{sp}}\,dx\le C 
\int_{B(z,8r)} \int_{B(z,8r)} \frac{\lvert u(x)-u(y)\rvert^p}{d(x,y)^{sp}\, \mu(B(x,d(x,y)))}\,dy\,dx
\]
for all  $u\in \Lip(X)$ such that $u=0$ in $F$.
\end{theorem}

\begin{proof}
Let 
$u\in \Lip(X)$ such that $u=0$ in $F$.
By considering the Lipschitz functions
\[
\max\{0,\lvert u(x)\rvert-\tau\}\,,\qquad \tau>0\,,
\]
and applying Fatou's lemma, we may assume that $u=0$ in $F_\tau$ for some
$\tau>0$, where $F_\tau$ is the $\tau$-neighborhood of $F$.
 For $k\in \Z$, we
denote
\[
 E_k = \{x\in B(z,8r) : |u(x)| > 2^k \}
\]
and we define a function $u_k:X\to [0,1]$ by
\[
u_k(x) = \begin{cases}
1, & \text{if $|u(x)| \geq 2^{k+1}$,}\\
\frac{|u(x)|}{2^k}-1, &\text{if $2^k < |u(x)| < 2^{k+1}$,}\\
0, & \text{if $|u(x)| \leq 2^k$.}
\end{cases}
\]
Then $u_k \in \Lip(X)$ and it satisfies $u_k=1$ on $E_{k+1}$ and $u_k=0$ on $B(z,8r)\setminus E_k\supset F$.

We estimate
\begin{equation}\label{e.distcap}
\begin{split}
 \int_{B(z,8r)\setminus F}  \frac{\lvert u(x)\rvert^p}{\dist(x,F)^{sp}}\,dx
 &\leq
 \sum_{k\in \Z} 2^{(k+3)p} \int_{E_{k+2}\setminus E_{k+3}}  \frac{1}{\dist(x,F)^{sp}}\,dx \\&\leq  \sum_{k\in \Z} 2^{(k+3)p} \int_{E_{k+2}}  \frac{1}{\dist(x,F)^{sp}}\,dx\\&\leq  \sum_{k\in \Z} 2^{(k+3)p} \int_{B(z,8r)\setminus F}  \frac{u_{k+1}(x)^p}{\dist(x,F)^{sp}}\,dx\,.
\end{split}
\end{equation}
Fix $k\in\Z$. By the $5r$-covering lemma \cite[Lemma 1.7]{MR2867756}, there exists a cover of $B(z,2r)\setminus F$
by balls $B_i=B(x_i,r_i)$, $i\in\N$,  such that 
$\sum_{i=1}^\infty \mathbf{1}_{B_i}\le C(c_\mu)$ and
$x_i\in B(z,2r)\setminus F$ and $r_i=\dist(x_i,F)/2$
for every $i\in\N$.
We first estimate the integral over the smaller set $B(z,2r)\setminus F$ as follows
\begin{equation}\label{e.2b}
\begin{split}
\int_{B(z,2r)\setminus F}  \frac{u_{k+1}(x)^p}{\dist(x,F)^{sp}}\,dx&\le\sum_{i=1}^\infty
\int_{B(z,2r)\cap B_i}  \frac{u_{k+1}(x)^p}{\dist(x,F)^{sp}}\,dx\\&\le \sum_{i=1}^\infty r_i^{-sp} \int_{B(z,2r)\cap B_i} u_{k+1}(x)^p\,dx\,.
\end{split}
\end{equation}
For every $x\in X$ we write
\begin{align*}
g_k(x)&=\mathbf{1}_{B(z,8r)}(x)\left(\int_{B(z,8r)} \frac{\lvert u_k(x)-u_k(y)\rvert^p}{d(x,y)^{sp}\, \mu(B(x,d(x,y)))}\,dy\right)^{\frac{1}{p}}\,,\\
h_k(x)&=\mathbf{1}_{B(z,8r)\setminus F}(x)\frac{u_{k-1}(x)}{\dist(x,F)^s}\,.
\end{align*}
Fix $i\in\N$ and fix a point $x\in B(z,2r)\cap B_i$. 
We  aim to  show that
\begin{equation}\label{e.des}
u_{k+1}(x)^p\le C(\delta,s,p,c_\mu)r_i^{sp}\left(M(g_k^q)(x)\right)^{p/q} + \varepsilon^{p/q} r_i^{sp}(M(h_k^q)(x))^{p/q}\,,
\end{equation}
where $\varepsilon>0$ and $\delta>0$ are as in the statement
of Lemma \ref{l.bigpiece} for $\alpha=sq$. We will choose these parameters later.
Clearly, we may assume that $u_{k+1}(x)\not=0$.
Fix $m\in\N_0$ such that
$2^{-m}r\le \dist(x_i,F)<2^{-m+1}r$. Write $B_i^*=B(x_i,2^{-m+2}r)$. Observe
that $x\in B_i\subset B_i^*$.

 There are two cases:

{\bf 1. $(u_k)_{B_i^*}<1-\delta$.} Since $u_{k+1}(x)\not=0$ and $x\in B(z,2r)$, we have
$x\in E_{k+1}$ and $u_k(x)=1$. Therefore
\begin{equation}\label{e.p1}
1=\frac{1}{\delta}(1-(1-\delta))<\frac{1}{\delta}(u_k(x)-(u_k)_{B_i^*})=\frac{1}{\delta}\lvert u_k(x)-(u_k)_{B_i^*}\rvert\,.
\end{equation}
Denote  $B'_j=B(x,2^{-j-m}r)$ for each $j\in\N_0$. We estimate
\begin{align*}
\lvert u_k(x)-(u_k)_{B_i^*}\rvert\le \lvert u_k(x)-(u_k)_{B'_0}\rvert + \lvert (u_k)_{B'_0}-(u_k)_{B_i^*}\rvert.
\end{align*}
Observe that $x\in B_j'\subset B_0'\subset B_i^*\subset B(z,8r)$ for all $j\in\N_0$.
Hence, by the $(s,1,q,p)$-Poincar\'e inequality \cite[Lemma 2.2]{DLV2021}, we obtain
\begin{equation}\label{e.p21}
\begin{split}
\lvert (u_k)_{B'_0}-(u_k)_{B_i^*}\rvert &\le \vint_{B'_0} \lvert u_k(y)-(u_k)_{B_i^*}\rvert\,dy 
\le c_\mu^3\vint_{B_i^*} \lvert u_k(y)-(u_k)_{B_{i}^*}\rvert\,dy \\
&\le C(s,p,c_\mu)(2^{-m+2}r)^s \biggl(\vint_{B_{i}^*}\biggl(\int_{B_{i}^*} \frac{\lvert u_k(y)-u_k(w)\rvert^p}{d(y,w)^{sp}\, \mu(B(y,d(y,w)))}\,dw\,\biggr)^{q/p}dy\biggr)^{1/q}\\
&\le C(s,p,c_\mu)(2^{-m+2}r)^s \biggl(\vint_{B_{i}^*}g_k(y)^q\,dy\biggr)^{1/q}\\
&\le C(s,p,c_\mu)r_i^s (M(g_k^q)(x))^{1/q}\,.
\end{split}
\end{equation}
Here $M(g_k^q)$ is the uncentered maximal function of $g_k^q$.
Furthermore, 
 \begin{equation}\label{e.p2}
\begin{split}
&\lvert u_k(x)-(u_k)_{B'_0}\rvert=\lim_{j\to\infty} \lvert (u_k)_{B'_{j+1}}-(u_k)_{B'_0}\rvert
\le \sum_{j=0}^\infty \lvert (u_k)_{B'_{j+1}}-(u_k)_{B'_{j}}\rvert\\
&\quad \le \sum_{j=0}^\infty  c_\mu\vint_{B'_j} \lvert u_k(y)-(u_k)_{B'_{j}}\rvert\,dy \\
&\quad\le C(s,p,c_\mu)\sum_{j=0}^\infty (2^{-j-m}r)^s \biggl(\vint_{B'_{j}}\biggl(\int_{B'_{j}} \frac{\lvert u_k(y)-u_k(w)\rvert^p}{d(y,w)^{sp}\, \mu(B(y,d(y,w)))}\,dw\,\biggr)^{q/p}dy\biggr)^{1/q}\\
&\quad\le C(s,p,c_\mu) r_i^s \sum_{j=0}^\infty 2^{-js} \left(\vint_{B_j'} g_k(y)^q\,dy\right)^{1/q}\\
&\quad\le  C(s,p,c_\mu) r_i^s (M(g_k^q)(x))^{1/q}\sum_{j=0}^\infty 2^{-js} \le C(s,p,c_\mu) r_i^s (M(g_k^q)(x))^{1/q}\,.
\end{split}
\end{equation}
Hence, from \eqref{e.p1}, \eqref{e.p21} and \eqref{e.p2}, we can conclude that
\[
u_{k+1}(x)^p\le 1\le C(\delta,s,p,c_\mu)r_i^{sp}\left(M(g_k^q)(x)\right)^{p/q}
\]
and \eqref{e.des} follows in the first case.

{\bf 2. $(u_k)_{B_i^*}\ge 1-\delta$.} We have $B_i^*\subset B(z,8r)$ and $u_k=0$ on $B(z,8r)\setminus E_k$, and therefore
\[
1-\delta\le (u_k)_{B_i*}= \vint_{B_i^*} u_k(y)\,dy\le \vint_{B_i^*} \mathbf{1}_{E_k}(y)\,dy\,.
\]
This gives
\[
\vint_{B_i^*} \mathbf{1}_{X\setminus E_k}(y)\,dy=1-\vint_{B_i^*} \mathbf{1}_{E_k}(y)\,dy\le 1-(1-\delta)=\delta\,.
\]
Recall that $\varepsilon$ and $\delta$ satisfy the assumption of Lemma \ref{l.bigpiece} with $\alpha=sq$. Since $u_{k-1}=1$ on $E_k$ and $x\in B_i^*\subset B(z,8r)$, we have
\begin{align*}
1&\le \varepsilon (2^{-m-1}r)^{sq}\vint_{B_i^*} \frac{\mathbf{1}_{E_k}(y)}{\dist(y,F)^{sq}}\,dy
\\&\le \varepsilon r_i^{sq}\vint_{B_i^*} \mathbf{1}_{B(z,8r)\setminus F}(y)\frac{u_{k-1}(y)^q}{\dist(y,F)^{sq}}\,dy\le \varepsilon r_i^{sq}M(h_k^q)(x)\,.
\end{align*}
By raising both sides to power $p/q>1$, we get
\[
u_{k+1}(x)^p\le 1\le \varepsilon^{p/q} r_i^{sp}(M(h_k^q)(x))^{p/q}
\]
and \eqref{e.des} follows also in the second case.

By combining inequalities \eqref{e.2b} and \eqref{e.des}, and since $\sum_{i=1}^\infty\mathbf{1}_{B_i}\le C(c_\mu)$, we get
\begin{align*}
&\int_{B(z,2r)\setminus F}  \frac{u_{k+1}(x)^p}{\dist(x,F)^{sp}}\,dx\le \sum_{i=1}^\infty r_i^{-sp} \int_{B(z,2r)\cap B_i} u_{k+1}(x)^p\,dx\\
&\quad\le \sum_{i=1}^\infty r_i^{-sp}\int_{B(z,2r)\cap B_i} C(\delta,s,p,c_\mu)r_i^{sp}\left(M(g_k^q)(x)\right)^{p/q}\,dx  \\
&\qquad\qquad\qquad\qquad   + \sum_{i=1}^\infty r_i^{-sp}\int_{B(z,2r)\cap B_i}\varepsilon^{p/q} r_i^{sp}(M(h_k^q)(x))^{p/q}\,dx\\
&\quad\le C(\delta,s,p,c_\mu)\sum_{i=1}^\infty \int_{B(z,2r)\cap B_i} \left(M(g_k^q)(x)\right)^{p/q}\,dx  \\
&\qquad\qquad\qquad\qquad   + \varepsilon^{p/q}\sum_{i=1}^\infty \int_{B(z,2r)\cap B_i} (M(h_k^q)(x))^{p/q}\,dx\\
&\quad\le C(\delta,s,p,c_\mu)\int_{X} \left(M(g_k^q)(x)\right)^{p/q}\,dx  \\
&\qquad\qquad\qquad\qquad   + C(c_\mu)\varepsilon^{p/q}\int_{X} (M(h_k^q)(x))^{p/q}\,dx\,.
\end{align*} 
Since $p/q>1$, by the Hardy--Littlewood maximal function theorem \cite[Theorem 3.13]{MR2867756}, 
\begin{align*}
\int_{B(z,2r)\setminus F}  \frac{u_{k+1}(x)^p}{\dist(x,F)^{sp}}\,dx &\le C_3\int_{X} g_k(x)^p\,dx 
 + C_4\varepsilon^{p/q}\int_{X} h_k(x)^{p}\,dx\\
 &= C_3\int_{B(z,8r)} g_k(x)^p\,dx+C_4\varepsilon^{p/q}\int_{B(z,8r)\setminus F} \frac{u_{k-1}(x)^p}{\dist(x,F)^{sp}}\,dx\,.
\end{align*}
where $C_3=C(\delta,s,p,q,c_\mu)$ and $C_4=C(p,q,c_\mu)$. 
Combining this estimate with Lemma \ref{l.absorb} yields
\begin{align*}
\int_{B(z,8r)\setminus F}  \frac{u_{k+1}(x)^p}{\dist(x,F)^{sp}}\,dx &\le 
C_1\int_{B(z,8r)} g_{k+1}(x)^p\,dx+
C_2\int_{B(z,2r)\setminus F}  \frac{u_{k+1}(x)^p}{\dist(x,F)^{sp}}\,dx
\\
&\le 
  C_1\int_{B(z,8r)} g_{k+1}(x)^p\,dx+C_2C_3\int_{B(z,8r)} g_k(x)^p\,dx\\
  &\qquad\qquad \qquad +C_2C_4\varepsilon^{p/q}\int_{B(z,8r)\setminus F} \frac{u_{k-1}(x)^p}{\dist(x,F)^{sp}}\,dx\,,
\end{align*}
where $C_1=C(s,p,c_\mu)$ and $C_2=C(s,p,c_\mu)$.
Multiplying both sides by $2^{(k+3)p}$ and summing the resulting inequalities gives
\begin{equation}\label{e.basicsum}
\begin{split}
&\sum_{k\in \Z} 2^{(k+3)p}\int_{B(z,8r)\setminus F}  \frac{u_{k+1}(x)^p}{\dist(x,F)^{sp}}\,dx\\
&\qquad \qquad \le 
  C_1\sum_{k\in \Z} 2^{(k+3)p}\int_{B(z,8r)} g_{k+1}(x)^p\,dx+C_2C_3\sum_{k\in \Z} 2^{(k+3)p}\int_{B(z,8r)} g_k(x)^p\,dx\\
  &\qquad\qquad \qquad +C_2C_4\varepsilon^{p/q}\sum_{k\in \Z} 2^{(k+3)p}\int_{B(z,8r)\setminus F} \frac{u_{k-1}(x)^p}{\dist(x,F)^{sp}}\,dx\,.
  \end{split}
  \end{equation}
We estimate the three terms in the right-hand side.
First, by Lemma \ref{l.mazya}, we have
\begin{align*}
&\sum_{k\in \Z} 2^{(k+3)p}\int_{B(z,8r)} g_{k+1}(x)^p\,dx
\\&\quad \le 2^{2p}\sum_{k\in\Z}2^{(k+1)p}\int_{B(z,8r)}\int_{B(z,8r)}\frac{\lvert u_{k+1}(x)-u_{k+1}(y)\rvert^p}{d(x,y)^{sp}\, \mu(B(x,d(x,y)))}\,dy\,dx \\
&\quad \le C(p)\int_{B(z,8r)}\int_{B(z,8r)}\frac{\lvert u(x)-u(y)\rvert^p}{d(x,y)^{sp}\, \mu(B(x,d(x,y)))}\,dy\,dx\,.
\end{align*}
A similar argument shows that
\begin{align*}
&\sum_{k\in \Z} 2^{(k+3)p}\int_{B(z,8r)} g_k(x)^p\,dx
\\&\quad \le 2^{3p}\sum_{k\in\Z}2^{kp}\int_{B(z,8r)}\int_{B(z,8r)}\frac{\lvert u_{k}(x)-u_{k}(y)\rvert^p}{d(x,y)^{sp}\, \mu(B(x,d(x,y)))}\,dy\,dx \\
&\quad \le C(p)\int_{B(z,8r)}\int_{B(z,8r)}\frac{\lvert u(x)-u(y)\rvert^p}{d(x,y)^{sp}\, \mu(B(x,d(x,y)))}\,dy\,dx\,.
\end{align*}
We choose $\varepsilon>0$ such that $C_2C_42^{2p}\varepsilon^{p/q}\le 1/2$. 
By changing the summation variable, we see that the last term in \eqref{e.basicsum} can be estimated from above by
\begin{align*}
&C_2C_42^{2p}\varepsilon^{p/q}\sum_{k\in \Z} 2^{(k+1)p}\int_{B(z,8r)\setminus F} \frac{u_{k-1}(x)^p}{\dist(x,F)^{sp}}\,dx
\le \frac{1}{2}\sum_{k\in \Z} 2^{(k+3)p}\int_{B(z,8r)\setminus F} \frac{u_{k+1}(x)^p}{\dist(x,F)^{sp}}\,dx\,.
\end{align*}
Hence, by  absorbing the last term in \eqref{e.basicsum} to the left-hand side and using \eqref{e.distcap} gives the desired inequality.
We remark that the absorbed term is finite, since $u$ is bounded on $B(z,8r)$ and $u=0$ in $F_\tau$ for some
$\tau>0$.
\end{proof}

\section{Proof of the main result}\label{e.char}

Here we prove our main result, which is Theorem \ref{t.main}. We begin with two lemmata.

\begin{lemma}\label{e.easy}
Let $\alpha>0$.
Suppose that a set $E\subset X$ is such that $\ucodima(E)<\alpha$.
Then $E$ satisfies the upper Aikawa condition with exponent $\alpha$.
\end{lemma}

\begin{proof}
By definition of the Assouad codimension, we see that \eqref{e.uascd} holds for some exponent $Q<\alpha$ and a constant $c>0$. Let $\varepsilon>0$.
Fix $0<\theta<1$ so small that 
\[
\frac{1}{\varepsilon}\le \frac{c}{2}\theta^{Q-\alpha}
\]
Choose $\delta= \frac{c}{2}\theta^{Q}>0$.
Fix either $0<R<\diam(E)$ if $\diam(E)>0$, or $0<R<\diam(X)$ if $\diam(E)=0$.
Fix $z\in E$  and a Borel set $K\subset X$ such that
\[
\mu(B(z,R)\setminus K)\le \delta\mu(B(z,R))\,.
\]
Let $0<r=\theta R<R$, so that $\theta=r/R$.
By using \eqref{e.uascd}, we get
\begin{align*}
c\Bigl(\frac{r}{R}\Bigr)^Q&\le \vint_{B(z,R)} \mathbf{1}_{E_r}(y)\,dy\\
&=\vint_{B(z,R)} \mathbf{1}_{E_r\cap K}(y)\,dy+\vint_{B(z,R)} \mathbf{1}_{E_r\setminus K}(y)\,dy\\
&\le \vint_{B(z,R)} \mathbf{1}_{E_r\cap K}(y)\,dy+\vint_{B(z,R)} \mathbf{1}_{B(z,R)\setminus K}(y)\,dy\le \vint_{B(z,R)} \mathbf{1}_{E_r\cap K}(y)\,dy+\delta\,.
\end{align*}
Recall that
\[
\delta=\frac{c}{2}\theta^{Q}=\frac{c}{2}\left(\frac{r}{R}\right)^{Q}\,.
\]
Hence, absorbing $\delta$ to the left-hand side and using the estimate
$\dist(y,E)< r$ for $y\in E_r$ yields
\begin{align*}
\frac{c}{2}\Bigl(\frac{r}{R}\Bigr)^Q&\le \vint_{B(z,R)} \mathbf{1}_{E_r\cap K}(y)\,dy\\
&\le
r^{\alpha}\vint_{B(z,R)} \frac{\mathbf{1}_{E_r\cap K}(y)}{\dist(y,E)^{\alpha}}\,dy\le \left(\frac{r}{R}\right)^{\alpha} R^\alpha\vint_{B(z,R)} \frac{\mathbf{1}_{K}(y)}{\dist(y,E)^{\alpha}}\,dy
\end{align*}
All in all, we get
\begin{align*}
\frac{1}{\varepsilon}\le \frac{c}{2}\theta^{Q-\alpha}=\frac{c}{2}\Bigl(\frac{r}{R}\Bigr)^{Q-\alpha}\le R^\alpha \vint_{B(z,R)} \frac{\mathbf{1}_{K}(y)}{\dist(y,E)^{\alpha}}\,dy\,,
\end{align*}
which shows that $E$ satisfies the upper Aikawa condition with exponent $\alpha$.
\end{proof}

%
%

\begin{lemma}\label{t.difficult_converse}
Let
$1 <q<p<\infty$ and $0<s<1$. Assume that a  set $E\subset X$ satisfies the upper Aikawa condition with exponent 
$\alpha=sq$.
 Then
$\ucodima(E)\le sp$.
\end{lemma}

\begin{proof}
We first show that the closure $\iol{E}$ satisfies
the upper Aikawa condition with exponent $\alpha=sq$.
This is clear if $\diam(E)=0$, so we can
assume that $\diam(E)>0$.
We let $\varepsilon>0$
and denote $\varepsilon'=\varepsilon/c_\mu^2>0$.
Let $\delta'>0$ be such that $E$ satisfies
the upper Aikawa condition with the pair $\varepsilon'$ and $\delta'$.
Choose $\delta=\delta' / c_\mu^2$.
 Let $z\in\iol{E}$, $0<R<\diam(\iol{E})$ and assume
that a Borel set $K\subset X$ satisfies the condition
\[
\mu(B(z,R)\setminus K)\le \delta\mu(B(z,R))\,.
\]
Fix a point $w\in B(z,R/2)\cap E$. Then $w\in E$, $0<R/2<\diam(E)$  and
\begin{align*}
\mu(B(w,R/2)\setminus K)&\le \mu(B(z,R)\setminus K)\le
\delta\mu(B(z,R))\\&\le c_\mu^2\delta\mu(B(w,R/2))=\delta'\mu(B(w,R/2))\,.
\end{align*}
Therefore
\[
R^{-\alpha}\le (R/2)^{-\alpha}\le \varepsilon'\vint_{B(w,R/2)}\frac{\mathbf{1}_K(y)}{\dist(y,E)^\alpha}\,dy
\le \varepsilon' c_\mu^2\vint_{B(z,R)}\frac{\mathbf{1}_K(y)}{\dist(y,\iol{E})^\alpha}\,dy
\]
Since $\varepsilon'c_\mu^2=\varepsilon$, the closure $\iol{E}$ satisfies the upper Aikawa condition
with exponent $\alpha=sq$.

Next we adapt the proof of \cite[Theorem 5.3]{DLV2021}.
Fix $z\in E\subset \iol{E}$. We also fix either $0<r<R<\diam(E)$
if $\diam(E)>0$, or $0<r<R<\diam(X)$ if $\diam(E)=0$.
It suffices to show that
\begin{equation}\label{eq.codim_est_easy}
\frac{\mu(E_r\cap B(z,R))}{\mu(B(z,R))}\ge c\Bigl(\frac{r}{R}\Bigr)^{sp},
\end{equation}
where  the constant $c>0$ is independent of $z$, $r$ and $R$.

If $\mu(E_r\cap B(z,R))\ge \frac 1 2 \mu(B(z,R))$, the claim is clear since $\bigl(\frac{r}{R}\bigr)^{sp}\le 1$. Thus we may assume in the sequel that
$\mu(E_r\cap B(z,R))< \tfrac 1 2 \mu(B(z,R))$, whence
\begin{equation}\label{eq.compl_meas_easy}
\mu(B(z,R)\setminus E_r) \ge \tfrac 1 2 \mu(B(z,R))>0.
\end{equation}

We define a  Lipschitz function $u\in \Lip(X)$ by 
\[
u(x)=\min\{1,4r^{-1}\dist(x,E)\}\,,\qquad  x\in X.
\]
Then $u=0$ in $\iol{E}$.
Since $u=1$ in $X\setminus E_{r}$, we obtain
\begin{equation}\label{eq.lhs_low_easy}
\begin{split}
R^{-sp}\int_{B(z,R)} \lvert u(x)\rvert^p\,dx &\ge 
R^{-sp}\int_{B(z,R)\setminus E_{r}} \lvert u(x)\rvert^p\,dx
\\&= 
R^{-sp}\mu(B(z,R)\setminus E_{r})
\ge \tfrac 1 2 R^{-sp}\mu(B(z,R)),
\end{split}
\end{equation}
where the last step follows from~\eqref{eq.compl_meas_easy}.

Assume that $F\subset \iol{E}$ is as in Lemma \ref{l.truncation} 
applied with the closed set $\iol{E}$ and the ball $B(z,R/8)$. Then $z\in F$ and $u=0$ in $F$. Using inequalities \eqref{eq.lhs_low_easy}  and 
Theorem \ref{p.maz'ya_local} with the closed set $\iol{E}$, we get
\begin{equation}\label{eq.rhs_upper_easy}
\begin{split}
C_1\int_{B(z,R)} \int_{B(z,R)} \frac{\lvert u(x)-u(y)\rvert^p}{d(x,y)^{sp}\, \mu(B(x,d(x,y)))}\,dy\,dx
&\ge
2\int_{B(z,R)\setminus F}  \frac{\lvert u(x)\rvert^p}{\dist(x,F)^{sp}}\,dx 
\\&\ge
2 R^{-sp}\int_{B(z,R)} \lvert u(x)\rvert^p\,dx\\&\ge R^{-sp}\mu(B(z,R))\,,
\end{split}
\end{equation}
where the constant $C_1$ is independent of $z$, $r$ and $R$.
To prove the claim~\eqref{eq.codim_est_easy}, it hence suffices to show that
\begin{equation}\label{eq.rhs_upper}
\int_{B(z,R)} \int_{B(z,R)} \frac{\lvert u(x)-u(y)\rvert^p}{d(x,y)^{sp}\, \mu(B(x,d(x,y)))}\,dy\,dx
\le C_2 r^{-sp}\mu(E_r\cap B(z,R))\,,
\end{equation}
where the constant $C_2$ is independent of $z$, $r$ and $R$.
This estimate follows by arguing as in the proof of \cite[Theorem 5.3]{DLV2021}; we omit the somewhat
lengthy argument.
\end{proof}

\begin{theorem}\label{t.main}
Let $E\subset X$ be a nonempty set. Then $\ucodima(E)$
is the infimum of all $\alpha>0$ for which $E$ satisfies
the upper Aikawa condition with exponent $\alpha$.
\end{theorem}

\begin{proof}
Denote by $A$ the set of all $\alpha>0$ for which $E$ satisfies
the upper Aikawa condition with exponent $\alpha$. Lemma~\ref{e.easy} shows that
$A\not=\emptyset$, and more specifically $(\ucodima(E),\infty)\subset A$. Therefore $\inf A\le \ucodima(E)$. Conversely, 
assume that $\alpha\in A$ so that $E$ satisfies the upper Aikawa condition with exponent $\alpha>0$. Choose $1<q<\infty$ and $0<s<1$ such that $\alpha=sq$. Fix
$p>q$. Lemma~\ref{t.difficult_converse} implies that
$\ucodima(E)\le sp$. Since $p>q$ was arbitrary, it follows that
$\ucodima(E)\le sq=\alpha$. Since $\alpha\in A$ was arbitrary, we find that
$\ucodima(E)\le \inf A$.
\end{proof}

\def\cprime{$'$} \def\cprime{$'$} \def\cprime{$'$}

\end{document}